\documentclass[a4paper,12pt]{extarticle}
\usepackage{titlesec} 
\usepackage{amsfonts}
\usepackage{amsthm}
\usepackage{amssymb}
\usepackage{amsmath}
\usepackage{theoremref}
\usepackage{enumerate}
\usepackage{bbm}
\usepackage{xcolor}
\usepackage{mathtools}

\DeclarePairedDelimiter\abs{\lvert}{\rvert}%
\DeclarePairedDelimiter\norm{\lVert}{\rVert}%

\makeatletter
\let\oldabs\abs
\def\abs{\@ifstar{\oldabs}{\oldabs*}}
\let\oldnorm\norm
\def\norm{\@ifstar{\oldnorm}{\oldnorm*}}
\makeatother

\makeatletter
\g@addto@macro\bfseries{\boldmath}
\makeatother

\newcommand{\A}{\mathcal{A}}
\newcommand{\C}{\mathcal{C}}
\newcommand{\M}{\mathcal{M}}
\newcommand{\K}{\mathcal{K}}
\newcommand{\N}{\mathcal{N}}
\newcommand{\T}{\mathbb{T}}

\newcommand{\z}{\zeta}
\newcommand{\conj}[1]{\overline{#1}}
\newcommand{\D}{\mathbb{D}}

\newcommand{\Po}{\mathcal{P}}
\newcommand{\cD}{\conj{\mathbb{D}}}
\newcommand{\ip}[2]{\big\langle #1, #2 \big\rangle}
\newcommand{\dist}[2]{\text{dist} ( #1, #2) }

\newcommand{\m}{\textit{m}}
\newcommand{\hil}{\mathcal{H}} 
\newcommand{\hd}{Hol(\D)}

\newcommand{\hb}{\mathcal{H}(b)}

\newcommand{\h}{\mathcal{H}}

\newcommand{\supp}[1]{\text{supp}({#1})}

\newtheorem{thm}{Theorem}[section]
\newtheorem{lemma}[thm]{Lemma}
\newtheorem{cor}[thm]{Corollary}
\newtheorem{prop}[thm]{Proposition}
\theoremstyle{definition}

\theoremstyle{definition}

\newtheorem*{remark}{Remark}

\titleformat{\section}
{\normalfont\scshape\centering}{\thesection}{1em}{}

\titleformat{\subsection}
[runin]{\bfseries\scshape}{\thesubsection}{0.5em}{}

\newcommand{\Addresses}{{
		\bigskip
		\footnotesize
		
		Adem Limani, \\ \textsc{Centre for Mathematical Sciences, Lund University, \\
			Lund, Sweden}\\
		\texttt{adem.limani@math.lu.se}
		
		\medskip
		
		Bartosz Malman, \\ \textsc{KTH Royal Institute of Technology, \\
			Stockholm, Sweden}\\
			\texttt{malman@kth.se}
		
	}}

\begin{document}
\title{\textbf{Inner functions, invariant subspaces and cyclicity in $\Po^t(\mu)$-spaces}}
\date{ }
\author{Adem Limani and Bartosz Malman}
\maketitle

\begin{abstract} 
We study the invariant subspaces generated by inner functions for a class of $\Po^t(\mu)$-spaces which can be identified as spaces of analytic functions in the unit disk $\D$, where $\mu$ is a measure supported in the closed unit disk and $\Po^t(\mu)$ is the span of analytic polynomials in the usual Lebesgue space $L^t(\mu)$. Our measures define a range of spaces somewhere in between the Hardy and the Bergman spaces, and our results are thus a mixture of results from these two theories. For a large class of measures $\mu$ we characterize the cyclic inner functions, and exhibit some interesting properties of invariant subspaces generated by non-cyclic inner functions. Our study is motivated by a connection with the problem of smooth approximations in de Branges-Rovnyak spaces.
\end{abstract}

\section{Introduction}
\label{introsec}

Let $t>0$ and $\mu$ be a compactly supported positive finite Borel measure in the complex plane $\mathbb{C}$. Denote by $\Po^t(\mu)$ the closure of the analytic polynomials in the usual Lebesgue space $L^t(\mu)$. The classical Hardy spaces $H^t$ can be defined as the closure of analytic polynomials in $L^t(d\m)$, where $d\m$ is the Lebesgue measure on the unit circle $\T$. The Bergman spaces can likewise be defined as the closures of analytic polynomials in spaces $L^t(dA)$, where $dA$ is the area measure on the unit disk $\D$. In both cases, it so happens that these closures can be identified with spaces of analytic functions, but for general measures $\mu$ this is no longer true. For instance, if  $\mu$ is a measure supported on $\T$ and contains a singular part $\mu_s$, then $\Po^{2}(\mu)$ contains $L^2(\mu_s)$ as a direct summand, which is certainly not a space of analytic functions. 
For absolutely continuous measures supported on $\T$ this phenomenon is explained by a famous theorem of Szeg\H{o}. Namely, if $\omega$ is a positive integrable function on $\T$, then the space $L^2(\omega d\m)$ can be identified with a space of analytic functions if and only if $\int_\T \log(\omega) d\m > -\infty$. Moreover, if the log-integrability condition fails, then we actually have that $\Po^2(  \omega dm ) = L^2(\omega dm)$. 

In the present paper, we shall introduce a class of measures $\mu$ supported on the closed unit disk $\cD$ for which the space $\Po^t(\mu)$ can be identified with a sufficiently well-behaved space of analytic functions on $\D$. We will stay within the realm of Bergman-type spaces, but we will be working with spaces considerably larger than the usual Hardy spaces. Our purpose is to study subspaces invariant for the forward shift operator $\M_z : f(z) \to zf(z)$ which are generated by inner functions. Our work is motivated by a connection with certain approximation problems in the theory of de Branges-Rovnyak spaces which we will briefly explain below.

In the case $t=2$, the spaces $\Po^2(\mu)$ and the shift operator $\M_z$ are universal models for the class of subnormal operators (see \cite{conway1991theory}). A subnormal operator is the restriction to an invariant subspace of a normal operator on a Hilbert space, and any such operator which also has a cyclic vector is unitarily equivalent to $\M_z$ on $\Po^2(\mu)$ for some measure $\mu$ supported in the complex plane. The spaces $\Po^2(\mu)$ which will fall into our framework here will be the ones on which $M_z$ acts as a \textit{completely non-isometric} contraction, in the sense that the ability to identify our spaces as spaces of analytic functions is equivalent to non-existence of invariant subspaces of $M_z$ on which it acts as an isometry.

We will now survey a number of results on $\Po^t(\mu)$-spaces as a way of justifying our choice of the class of measures $\mu$ which will be presented below. If $\mu$ is supported on the closed unit disk $\cD$, then a necessary condition for $\Po^t(\mu)$ to be a space of analytic functions is that the space contains no characteristic functions $1_F$ for any measurable subset $F \subset \cD$. This property is sometimes called \textit{irreducibility}, as it was for instance in \cite{aleman2009nontangential}. The case we will be studying is when the restriction $\mu|\D$ of $\mu$ to the open disk $\D$ already makes $\Po^t(\mu|\D)$ into a space of analytic functions on $\D$. Then the possibility to identify $\Po^t(\mu)$ with a space of analytic functions is equivalent to non-existence of a sequence of analytic polynomials $p_n$ such that 
\[ 
\lim_{n \to \infty} \int_\D |p_n|^t d\mu|\D = 0
\] 
and 
\[ 
\lim_{n \to \infty} \int_\T |p_n - 1_F|^t d\mu|\T = 0
\] 
for some measurable subset $F \subseteq \T$, $\mu|\T$ being the restriction of $\mu$ to $\T$. We will see that this condition is equivalent to $M_z: \Po^t(\mu) \to \Po^t(\mu)$ admitting no invariant subspace on which it acts isometrically. In the literature, the phenomenon above is referred to as \textit{splitting} (at least when $F = \T$). Vast amount of research has been carried out investigating when splitting can occur, but the exact mechanism is still not understood. For instance, splitting has been shown to occur if the measure $\mu|\D$ decays extremely fast close to the boundary and $\mu|\T = \omega d\m$ with $\log(\omega) \not\in L^1(\T)$, see for instance \cite{kriete1990mean}. However, there exists an example in \cite{kegejanex} which produces a compact subset $F$ of $\T$ such that $\Po^2(dA + 1_F d\m)$ splits, even though, clearly, the measure $dA$ does not decay at all. An important result is \cite[Theorem D]{kriete1990mean}: if $\mu|\D$ does not decay too fast and if $\mu|\T = \omega d\m$ with $\int_I \log(\omega) d\m > -\infty$ for some circular arc $I$, then the situation above, with the two limits simultaneously being zero, cannot occur for any Borel set $F$ contained in $I$. Thus in this case, if $\omega$ is supported on $I$, then $\Po^2(\mu)$ becomes a space of analytic functions. 

The weight $\omega$ will be living on sets more general than intervals, as in the already cited result from \cite{kriete1990mean}. For a closed set $E$ on $\T$, we consider the complement $\T \setminus E = \cup_{k} A_k$, where $A_k$ are disjoint open circular arcs. The set $E$ is a \textit{Beurling-Carleson set} if the following condition is satisfied: 
\begin{gather} \label{carlesoncond}
\sum_k |A_k|\log \left( \frac{1}{ |A_k| } \right) < \infty. 
\end{gather}
Notation $|A_k|$ stands for the Lebesgue measure of the arc $A_k$.
The sets satisfying \eqref{carlesoncond} will appear in two contexts below, and we will need to consider both Beurling-Carleson of positive Lebesgue measure, and of zero Lebesgue measure. Based on the discussion and results stated above, we will introduce the following conditions on the positive finite measure $\mu$ supported on $\cD$.

\begin{enumerate}


\item[(A)] The measure $d\mu|\T = \omega d\m$ is absolutely continuous and non-zero only on a set $E$ which can be written as a countable union $E = \cup_k E_k$ of Beurling-Carleson sets of positive measure for which we have that
\[ 
\int_{E_k} \log(\omega) d\m > -\infty.
\]

\item[(B)] The measure $\mu|\D$ is a Carleson and a reverse Carleson measure for some standard weighted Bergman spaces, and thus satisfies the following Carleson embeddings:
\[
\Po^{t}\left((1-|z|)^{\beta}dA \right)  \hookrightarrow \Po^{t}(\mu|\D) \hookrightarrow \Po^{t}\left( (1-|z|)^{\alpha} dA\right)
\]
for some $-1 < \beta \leq \alpha$.
\end{enumerate}
The condition (A) is present primarily to make $\Po^t(\mu)$ into a space of analytic functions, but is alone not sufficient without some assumption on $\mu|\D$, and indeed it is an interplay between the assumption (A) and (B) that ensures that $\Po^t(\mu)$ is a space of analytic functions. On the other hand, a remarkable result by Khrushchev says that if $\omega$ is concentrated on a set which contains no Beurling-Carleson set of positive measure, then $\Po^t( d\mu|\D + \omega dm)$ splits for any standard weighted Bergman space $\Po^t(\mu|\D)$ (see \cite{khrushchev1978problem}, Theorem 4.1). This indicates that Beurling-Carleson sets of positive measure are the appropriate carriers of weights on the unit circle in order to ensure $\Po^t(\mu)$ being a space of analytic functions contained in the standard weighted Bergman spaces.
Beyond this first condition, we shall also implicitly assume that $\log(\omega)$ is not integrable on the entire unit circle $\T$, since integrability of this quantity leads to a space $\Po^t(\mu)$ which is essentially a Hardy space. The key point of assumption (B) is that it is readily satisfied by any standard Bergman weight $d\mu|\D =(1-|z|)^{\alpha}dA$, for $\alpha > -1$, and the reader can think of $\mu|\D$ as such a weight. Our techniques allow however to replace the standard Bergman weight by such a two-sided Carleson measure, which is a somewhat more general condition.

We move on to the presentation of our main results. As a starting point, we must of course establish that $\Po^t(\mu)$ can genuinely be identified as a space of analytic functions. This is a consequence of \thref{irreducibilitypropcarleson} which is proved in Section \ref{irreducibilitysection}. In fact, as soon as this is established, this result can be used in combination with the work of Aleman, Richter and Sundberg in \cite{aleman2009nontangential} to infer further properties of the space $\Po^t(\mu)$. This combination establishes the following result.

\begin{thm}\thlabel{irredTheorem}
Let $\mu = \mu|\D + \mu|\T$ be a positive finite measure on $\cD$ satisfying the assumptions $(A)$ and $(B)$. For any $t \in (0,\infty)$, the space $\Po^t(\mu)$ is a genuine space of analytic functions: the restriction $f|\D$ of any element $f \in \Po^t(\mu)$ is an analytic function on $\D$ and $f|\D$ extends uniquely to the element $f$ in $\Po^t(\mu)$. The space satisfies the following additional properties:
\begin{enumerate}[(i)]
\item Each $f \in \Po^t(\mu)$ has non-tangential limits $\mu|\T$-almost everywhere on $\T$. 
\item The index of each $M_z$-invariant subspace of $\Po^t(\mu)$ is equal to 1.
\end{enumerate}
\end{thm}
The \textit{index} notion in the theorem refers to the co-dimension of the image of $M_z$ when restricted to an invariant subspace. Our main motivation for this work was the study of the closed $M_z$-invariant subspaces of $\Po^t(\mu)$-spaces generated by inner functions. For a positive singular measure $\nu$ supported on $\T$, the \textit{singular inner function} $S_\nu$ is defined by the formula \begin{equation} \label{SingIn} S_\nu(z) = \exp\Big(-\int_\T \frac{\zeta + z}{\zeta - z} d\mu(\z)\Big). 
\end{equation}
A general inner function $\theta$ is the product $\theta = BS_\nu$ of a singular inner function $S_\nu$ and a convergent Blaschke product $B$ of the form \[ B(z) = \prod_{n=1}^\infty \frac{|a_n|}{a_n} \frac{a_n-z}{1-\conj{a_n}z}.
\]  
One of the main questions we want to answer is: for which inner functions $\theta$ are the polynomial multiples of $\theta$ dense in the space $\Po^t(\mu)$? We compare to the situation in Hardy and Bergman spaces. A classical result of operator theory characterizes the $M_z$-invariant subspaces of the Hardy space $H^p$ with $0<p<\infty$ as subspaces of the form $\theta H^2$. In particular, every inner function $\theta$ generates an invariant subspace properly contained in $H^p$. A description of invariant subspaces generated by inner functions exists also for the Bergman spaces, and it involves the notion Beurling-Carleson sets of zero measure, defined as above in \eqref{carlesoncond}. Every positive singular measure $\nu$ on $\T$ can be uniquely decomposed into a sum \begin{equation}
\label{BCKdecomp} \nu = \nu_\C + \nu_\K \end{equation} of mutually orthogonal measures where $\nu_\C$ is concentrated on a countable union of Beurling-Carlesons sets of measure zero, while $\nu_\K$ vanishes on every such Beurling-Carleson set. For a measure $\nu$, by \textit{concentrated on a set $F$}, we mean that $\mu(B) = \mu(B\cap F)$ for every Borel set $B$. The decomposition in \eqref{BCKdecomp} corresponds to a factorization $S_\nu = S_{\nu_\C}S_{\nu_\K}$. We will call $\nu_\C$ for the \textit{Beurling-Carleson} part of $\nu$ and $\nu_\K$ for the \textit{Korenblum-Roberts} part of $\nu$, with similar designations for the singular inner functions constructed from these measures. The works of Roberts in \cite{roberts1985cyclic} and Korenblum in \cite{korenblum1981cyclic} independently establish that the singular inner function $S_{\nu_\K}$ is a cyclic vector for the $M_z$-operator on the standard weighted Bergman spaces, while $S_{\nu_\C}$ generates a proper invariant subspace. The subspace generated by a non-cyclic inner function satisfies a very special property which we will below denote by (P) and describe next.

Let $\N^+$ denote the Smirnov class $f = u/v$ of quotients of bounded analytic functions, with $v$ being outer (see \cite{garnett}). Consider a pair $(X, \theta)$, where $X$ is a space of analytic function on $\D$ which contains bounded functions and on which the $M_z$ operator acts continuously, and $\theta$ be an inner function. Let $[\theta]_X$ be the smallest closed $M_z$-invariant subspace of $X$ containing $\theta$. Consider the following property of the pair $(X, \theta)$. 

\begin{enumerate}
\item[(P)] The invariant subspace $[\theta]_X$ satisfies $[\theta]_X \cap \N^+ \subseteq \theta \N^+$. \label{propertyP} \end{enumerate}

In other words, if $\theta h_n$ converges to $f$ in the topology of $X$, $h_n$ are polynomials and $f$ is contained in $\N^+$, then $f/\theta \in \N^+$. This phenomenon of \textit{permanence} of the inner factor has appeared implicitly in less general setting in the work of Roberts in \cite{roberts1985cyclic}. When $X$ is a standard weighted Bergman space, then then property (P) for $(X, \theta)$ is actually equivalent to density of classes of functions with some degree of regularity on $\T$ in the classical model spaces $K_\theta := H^2 \ominus \theta H^2$. Indeed, in \thref{BCpermanence} below, this permanence property for measures $\mu|\D$ satisfying (B) and Beurling-Carleson inner functions is deduced from the density result of \cite{smoothdensektheta}.

To further illustrate the importance of the concept, we note that the fundamental and useful fact that the functions with continuous extensions to $\T$ are dense in model spaces, established by Aleksandrov in \cite{aleksandrovinv}, can be related to the permanence of the inner factor under weak-star convergence in the space of Cauchy transforms. The density of the same class of functions in de Branges-Rovnyak spaces $\hb$ established in \cite{dbrcont} can be tied to a similar phenomenon. Establishing such corresponding properties for $\Po^t(\mu)$-spaces is related instead to approximations by functions with certain degree of smoothness in $\hb$-spaces, and the authors elaborate on this in \cite{DBRpapperAdem}. In short, functions from the smoothness class $\h^\tau$ defined below in \eqref{Xbetadef} are dense in space $\hb$ if and only if for the measure $d\mu = (1-|z|)^{\tau - 1} dA + (1-|b|^2) d\m$, the space $\Po^2(\mu)$ is a space of analytic functions and the property (P) holds for $\Po^2(\mu)$ and the inner factor of $b$. Precisely this connection is the main motivation for the present research.

Our results are of the kind discussed in the above paragraph. Before stating them we will need to set some further notations and conventions. The symbols $\mu|\T, \mu|\D, \omega, E, E_k$ appearing in the assumptions (A) and (B) will be used throughout the paper to denote the related objects. Thus, $E = \cup_k E_k$ denotes the carrier set of the weight $d\mu|\T = \omega d\m$ on $\T$ (that is, the set on which $\omega$ is concentrated), and the sets $E_k$ are Beurling-Carleson sets of positive measure. Further, the support of the measure $\mu|\T$ will be denoted $F$. It could be larger than the carrier $E$, but $E = F$ if $E$ is a single Beurling-Carleson set, keeping in mind that finite unions of Beurling-Carleson sets, again are Beurling-Carleson sets. The notation $\| \cdot \|_{\mu,t}$ denotes the norm on the space $L^t(\mu)$ for $t \geq 1$, and for $t \in (0,1)$ the usual translation-invariant metric. All of our results are valid for $t \in [1,\infty)$, and some are valid also for $t \in (0,\infty)$, and this will be indicated. For a Borel set $S$ in the closed disk we denote by $\mu|S$ the restriction of $\mu$ to the set $S$. For a measure $\nu$, the measures $\nu_\C$ and $\nu_\K$ denote the Beurling-Carleson and Korenblum-Roberts parts respectively, as in \eqref{BCKdecomp}.

The following theorem establishes the inner factor permanence property (P) for a large class of inner functions in $\Po^t(\mu)$-spaces. 
 
\begin{thm} \thlabel{InnerPermanenceTheorem}
Let $t \in [1,\infty)$ and $\mu$ satisfy (A) and (B), and let $\theta = BS_\nu$ be an inner function. Let $\nu_\K|E$ be the restriction of $\nu_\K$ to $E$. If $\theta_0 = BS_{\nu_\C}S_{\nu_{\K}|E}$, then the pair $(\Po^t(\mu), \theta_0)$ satisfies property (P), and thus generates a proper $M_z$-invariant subspace.
\end{thm}

The difference from the Bergman space case is that in our setting, the singular inner function constructed from the measure $\nu_\K|E$, which lives on the carrier of $\mu|\T$, satisfies the property (P) with respect to $\Po^t(\mu)$. In the Bergman spaces, this inner function is cyclic. 

Next, we present our result on cyclicity of inner functions in $\Po^t(\mu)$.


\begin{thm} \thlabel{cyclicitytheorem}
Let $\mu$ satisfy (B) and let $F$ denote the support of the measure $\mu|\T$. Suppose $S_\nu$ is a Korenblum-Roberts inner function, in the sense that $\nu$ vanishes on Beurling-Carleson sets of zero Lebesgue measure. If $\nu$ is concentrated on the complement of $F$, then $S_\nu$ is a cyclic vector for $M_z$ on $\Po^{t}(\mu)$, for all $t>0$.
\end{thm}

Note specifically that the assumption (A) in \thref{cyclicitytheorem} is not necessary, and that $F$ refers to the (closed) support of the measure $\mu$, which can indeed be larger than the set $E = \cup_{n} E_n$ appearing in condition (A). However, as a significant upside, the theorem applies to sets $F$ which are way more general than Beurling-Carleson sets. 


Combining \thref{InnerPermanenceTheorem} and \thref{cyclicitytheorem}, we can formulate a corollary which actually gives us complete description of invariant subspaces generated by inner functions, according to property (P) and cyclicity, for a large class of measures $\mu$. In fact, since outer functions are easily seen to be cyclic vectors in the spaces considered, the below result can be also considered as a description of invariant subspaces generated by bounded analytic functions. 

\begin{cor} \thlabel{maintheorem}
Let $t \in [1,\infty)$ and $\mu$ satisfy conditions (A) and (B) with $\mu|\T$ living on a single Beurling-Carleson set $E$. Let $\theta = BS_\nu$ be an inner function and decompose $\nu$ as \[ \nu = \nu_\C + \nu_\K|E + \nu_\K|E^c, \] where $E^c$ is the complement of $E$ with respect to $\T$. 
\begin{enumerate}[(i)]

\item If $\theta_0 = BS_{\nu_\C}S_{\nu_{\K}|E}$, then $\theta_0$ generates a proper $M_z$-invariant subspace and the pair $(\Po^t(\mu), \theta_0)$ satisfies property (P).

\item $S_{\nu_{\K}|E^c}$ is a cyclic vector for the shift operator $M_z$ on $\Po^t(\mu)$. 
\end{enumerate}
\end{cor}
Note that the result allows us to decompose every inner function into two factors, one satisfying the property (P), and one being cyclic. In particular, one easily deduces from the above that the invariant subspace generated by a bounded analytic function $f = \theta F$, with $F$ outer, equals the invariant subspace $[\theta_0]$. The result has an obvious Hardy-Bergman interplay flavour. Close to the support of the measure $\nu$ on the circle, the properties of the invariant subspace generated by an inner function $S_\nu$ are Hardy-like, while away from this support, the situation resembles what happens in the Bergman spaces. 

The rest of the paper is structured as follows. In Section \ref{irreducibilitysection} we will recall techniques from \cite{kriete1990mean} and \cite{khrushchev1978problem} and prove that our assumptions (A) and (B) allow us to identify $\Po^t(\mu)$ as a space of analytic functions on $\D$. In Section \ref{Cauchydualsec} we develop some duality theory which we will use in Section \ref{permanencesec}. Sections \ref{permanencesec} and \ref{cyclicsection} deal with proofs of our main theorems related to invariant subspaces generated by inner functions.

\section{Assumption (A), (B), and irreducibility of the shift}
\label{irreducibilitysection}

The below proposition is a version of a deep result on \textit{simultaneous approximation} by Khrushchev, found in \cite[Theorem 3.1]{khrushchev1978problem}. We shall introduce an alternate and minimal version of the result which is suitable for our further applications. We also give a sketch of the proof, which is essentially a minor modification of the original argument in \cite{khrushchev1978problem}. The proposition will be our principal technical tool in establishing that $M_z: \Po^t(\mu) \to \Po^t(\mu)$ is completely non-isometric if $\mu$ satisfies assumptions (A) and (B). In this section, we allow $t$ to be any finite positive number.

\begin{prop}\thlabel{hruscevKOtheorem}
Let $\mu|\T = \omega d\m$ be a measure satisfying condition (A) and which is concentrated on a single Beurling-Carleson set $E$ of positive measure. Let $\mu|\D$ be a measure satisfying the condition (B). Assume that there exists a sequence of analytic polynomials $\{p_n\}_{n \geq 1}$, an analytic function $f \in \Po^t(\mu|\D)$ and $g \in L^t(\omega d\m)$ such that \[ \lim_{n \to \infty} \int_\D |p_n-f|^t d\mu|\D = 0 \] and \[ \lim_{n \to \infty} \int_\T |p_n-g|^t \omega d\m = 0.\] If $g$ is non-zero, then $f$ is non-zero.
\end{prop}

\begin{proof}
As indicated above, we will follow closely the proof of \cite[Theorem 3.1]{khrushchev1978problem}. There, an auxiliary domain $G$ is constructed. If $\T \setminus E = \cup_n A_n$ are the complementary arcs to $E$, we associate to each arc $A_n$ its usual Carleson square $C_n$, and then the domain $G$ is obtained by removing the squares $C_n$ from $\D$. It is not difficult to see that the boundary $\partial G$ is a rectifiable curve, and that $\partial G \cap \T = E$. Let $\nu$ be the harmonic measure on $\partial G$ associated with the point $0 \in G$. Since assumption (B) is satisfied and thus $\Po^t(\mu|\D)$ is continuously contained in a standard weighted Bergman space, it follows that we can find constants $C,K > 0$, independent of $n$, such that 
\[
|p_n(z)| \leq C \|p_n\|_{\mu_\D, t} (1-|z|)^{-K}.
\]
Thus
\begin{gather} 
 \int_{\partial G} \log^+ |p_n| d\nu = \int_{\partial G \cap \D} \log^+(|p_n|) d\nu + \int_E \log^+(|p_n|) d\nu  \nonumber \\
 \lesssim 1 + \int_{\partial G \cap \D} \log^+\Big(\frac{1}{1-|z|}\Big) d\nu +  \int_E \log^+(|p_n|) d\nu . \label{s3eq2}
\end{gather}
The main technical difficulty of the proof of \cite[Theorem 3.1]{khrushchev1978problem} is using the assumption that $E$ is a Beurling-Carleson in order to show that the integral over $\partial G \cap \D$ is finite. For the proof of this fact, we refer the reader to the cited paper. Proceeding, assumption (A) ensures that there exists an outer function $h$ which satisfies $|h|^t = \omega$ on $E$, and from this we obtain 
\begin{equation} \label{s3eq1} 
\int_E \log^+(|p_n|) d\nu \lesssim \int_E \log^+(|p_n h|) d\nu + \int_E \log^+(|h|^{-1}) d\nu.
\end{equation} 
We will now show that the Lebesgue measure $dm$ dominates $d\nu$ on $E$. For any positive continuous function $\psi$ on $\T$, let $P \psi$ denote its Poisson extension to $\D$, and note that \[ \int_\T \psi d\m = P \psi (0) = \int_{\partial G} P \psi d\nu \geq \int_E P \psi d\nu = \int_E \psi d\nu.\] Thus $d\m - d\nu|E$ is a non-negative measure, and so \eqref{s3eq1} implies that 
\begin{equation} \label{s3eq3}
\int_E \log^+(|p_n|) d\nu \lesssim 1 + \int_E |p_n|^t \omega d\m + \int_E |\log \omega| d\m
\end{equation} which, by condition (A) and the convergence of the sequence $\{p_n\}_n$ in $L^t(\omega d\m)$, is a uniformly bounded quantity. Combining \eqref{s3eq2}, \eqref{s3eq1} and \eqref{s3eq3}, we get that \[ \sup_{n} \int_{\partial G} \log^+(|p_n|) d\nu < \infty. \] Let $\phi : \D \to G$ be a Riemann mapping satisfying $\phi(0) = 0$. It extends to a homeomorphism between $\T$ and $\partial G$. A change of variable on the previous expression implies that 
\begin{equation} \label{sec3eq4}
 \sup_{n} \int_{\T} \log^+(|p_n \circ \phi|) d\m < \infty.
\end{equation}
Because $p_n \circ \phi$ is continuous up to the boundary of $\D$, \eqref{sec3eq4} implies that $p_n \circ \phi$ is a bounded sequence in the Nevanlinna class of the unit disk (see \cite[Chapter II, Section 5]{garnett}). By modifying $g$ on a set of Lebesgue measure zero, we can assume that $g$ is a non-zero Borel measurable function. Extracting a subsequence, we may assume that the $p_n \to g$ almost everywhere on $E$, and by taking away a set of Lebesgue measure zero, we can construct a set $E_0$ which is a Borel set, $|E| = |E_0|$ and $p_n \to g$ everywhere on $E_0$. Note that $E_0$ is a subset of $\T \cap \partial G$ and so both $\nu(E_0)$ and $|E_0|$ make sense. Since $\partial G$ is rectifiable, it follows that $\nu$ is absolutely continuous with respect to the Lebesgue measure, and thus $\nu(E_0) > 0$. 

The function $g \circ \phi$ is measurable since $g$ is assumed to be a Borel function and $\phi$ is a continuous function. Thus $p_n \circ \phi \to g \circ \phi$ everywhere on the Borel set $\tilde{E_0} := \phi^{-1}(E_0)$ which satisfies $|\tilde{E_0}| = \nu(E_0) > 0$. Of course, $p_n \circ \phi \to f \circ \phi$ uniformly on compact subsets of $\D$, and $f \circ \phi$ is in the Nevanlinna class by \eqref{sec3eq4} and the subsequent remark. We are in the situation of the classical Khinchin-Ostrowski theorem (see \cite[Section 2 of Chapter 3]{havinbook}), and its application leads us to the following conclusion: the sequence $p_n \circ \phi$ converges on $\tilde{E_0}$ to the boundary values of $f \circ \phi$. Thus $f$ cannot be zero, since $p_n \circ \phi$ converges on $\tilde{E_0}$ to $g \circ \phi$, a non-trivial function on $\tilde{E_0}$.
\end{proof}

With the above technical result at hand, we now prove the main result of this section.

\begin{prop} \thlabel{irreducibilitypropcarleson}
If $\mu$ satisfies properties (A) and (B), then the shift operator $M_z$ defined on $\Po^t(\mu)$ is a completely non-isometric operator, in the sense that there exists no non-trivial closed invariant subspace of $X \subset \Po^t(\mu)$ with the property that $f \in X$ implies that $\|M_zf\|_{\mu,t}  = \|f\|_{\mu,t}$.
\end{prop}

\begin{proof}
Since $\Po^t(\mu) \subseteq \Po^t(\mu|\D) \oplus L^t(\mu|\T)$, we can represent $f \in \Po^t(\mu)$ as a tuple $(f_\D, f_\T)$, where $f_\D \in \Po^t(\mu|\D)$ and $f_\T \in L^t(\mu|\T)$. Note that if $\|M^n_zf\|_{\mu,t} = \|f\|_{\mu,t}$ for all integers $n\geq1$, then we must necessarily have that $f_\D \equiv 0$ and thus $f = (0, f_\T)$. Since the tuple $(0, f_\T)$ is a limit of polynomial tuples of the form $(p,p)$ in the norm of $\Po^t(\mu|\D) \oplus L^t(\mu|\T)$, it is also a limit of such tuples in the space $\Po^t(\mu| \D) \oplus L^t(\mu|E)$, where $\mu|E$ is the restriction to a Beurling-Carleson set (of positive Lebesgue measure) $E \subseteq \T$ of the measure $\mu|\T$. It follows by \thref{hruscevKOtheorem} that $f_\T$ is zero almost everywhere with respect to $\mu|E$ for every Beurling-Carleson set. Thus, since $\mu|\T$ is concentrated on a union of Beurling-Carleson sets, $f_\T$ also is zero almost everywhere with respect to $\mu|\T$. Consequently, $M_z$ is completely non-isometric on $\Po^t(\mu)$. 
\end{proof}

As a consequence, $\Po^t(\mu)$ can be identified with a space of analytic functions on $\D$ and every function in $ \Po^t(\mu)$ is uniquely determined by its analytic restriction to $\D$. \thref{irredTheorem} of Section \ref{introsec} now follows. 

\section{Cauchy duals of $\Po^t(\mu)$-spaces}

\label{Cauchydualsec}

Let $\mu$ be a positive measure supported on $\cD$ and $G \in L^1(\mu)$. The Cauchy transform $C_{G\mu}$ is the analytic function on $\D$ given by the formula \[ C_{G\mu} (z) = \int_{\cD} \frac{G(\zeta)d\mu(\z)}{1-\conj{\zeta}z}, \quad z \in \D. \]
The following realization of the dual space of $\Po^t(\mu)$ is particularly well-suited for our purposes. Here and throughout this section $t \in [1,\infty)$ and $t'$ is the H\"older conjugate exponent such that $1/t + 1/{t'} = 1$, of course interpreting $t'$ as infinity if $t = 1$. 

\begin{prop} \thlabel{CauchyDualChar}
Let $g \in \hd$ and $f$ be a polynomial. Then 
\begin{equation} \Lambda_g(f) := \lim_{r \to 1} \int_\T f(r\zeta)\conj{g(r\zeta)} d\m(\zeta) \end{equation} extends to a bounded linear functional on $\Po^t(\mu)$ if and only if $g$ is of the form $g = C_{G\mu}$ for some $G \in L^{t'}(\mu)$. If this is the case, then we have that \begin{equation} \label{LambdaNorm}
\abs{\Lambda_g(f)} = \inf \{\|G\|_{\mu,t'} : g = C_{G\mu} \}.
\end{equation} Moreover, every bounded linear functional on $\Po^t(\mu)$ can be represented in this way. 
\end{prop}

\begin{proof}
First assume that $g$ equals the Cauchy transform $C_{G\mu}$ with $G \in L^{t'}(\mu)$. If $f$ is a polynomial, then in particular $f(rz) \to f(z)$ uniformly on $\cD$ as $r \to 1-$, and from this it follows that
\begin{gather*}
\Lambda_g(f) = \lim_{r \to 1-} \int_\T f(r\zeta)\conj{g(r\zeta)} d\m(\zeta) \\ = \lim_{r \to 1-} \int_\T \int_{\cD} f(r\zeta)\frac{\conj{G(z)}}{1-r\conj{\zeta}z} d\mu(z) d\m(\zeta) \\
= \lim_{r \to 1-} \int_{\cD} \conj{G(z)} \Big( \int_\T \frac{f(r\zeta)}{1- r\conj{\zeta} z} dm(\zeta)\Big) d\mu(z) \\
= \lim_{r \to 1-} \int_{\cD} \conj{G(z)} f(r^2 z) d\mu(z) \\ = \int_{\cD} \conj{G(z)} f(z) d\mu(z). 
\end{gather*}
This shows that any $g= C_{G\mu}$ with $G \in L^{t'}(\mu)$ induces a bounded linear functional $\Lambda_g$, with
\[
\abs{\Lambda_g(f)} \leq \|G\|_{\mu,t'} \|f\|_{\mu,t}.
\]
Conversely, given any bounded linear functional $\Lambda$ on $\Po^t(\mu)$, the Hahn-Banach theorem and the usual duality between $L^t(\mu)$ and $L^{t'}(\mu)$ implies that $\Lambda$ can be represented in the form 
\[ 
\Lambda(f) = \int_{\cD} f(z) \conj{G(z)} d\mu(z)
\] 
for some $G \in L^{t'}(\mu)$. Following the steps backward in the above computation, we easily deduce that 
\begin{align*}
\int_{\cD} f(z) \conj{G(z)} d\mu(z) = \lim_{r \to 1-} \int_\T f(r\zeta) \overline{ \int_{\cD} \frac{G(z)}{1-r\conj{z}\zeta} d\mu(z) } \, d\m(\zeta).
\end{align*} Thus $\Lambda(f) = \Lambda_g(f)$, with $g = C_{G\mu}$. The norm equality in \eqref{LambdaNorm} follows in a standard way. 
\end{proof}

Our assumption (B) on the measure $\mu|\D$ implies that functions with sufficiently smooth extensions to the boundary correspond to bounded linear functionals on $\Po^t(\mu|\D)$ under Cauchy duality. Consequently, they are also bounded on the smaller space $\Po^t(\mu)$ (if the latter is a space of analytic functions). This observation is the content of our next corollary. For $\tau > 0$, let $\hil^\tau$ denote the so-called analytic Sobolev spaces:
\begin{equation} \label{Xbetadef}
\hil^\tau = \{ f \in \hd : \| f\|^2_{\hil^\tau} := \sum_{n=0}^\infty {(n+1)}^\tau |f_n|^2 < \infty \}.
\end{equation}
It is easy to see that given any positive integer $N$, there exists a $\tau > 0$ such that $\hil^\tau$ will consist entirely of functions for which the first $N$ derivatives extend continuously to the boundary of $\D$. 

Recall the quantity $\alpha>-1$ defined in property (B).

\begin{cor} \thlabel{CauchySmoothSolutions} If $\mu$ satisfies the second of the embeddings in (B), then for all $g \in \hil^\tau$ with $\tau = 4 + 2\alpha$, the functional $\Lambda_g$ is bounded on $\Po^t(\mu|\D)$. In particular, for any $g \in \hil^\tau$ there exists a solution $G \in L^{t'}(\mu|\D)$ to the equation $g = C_{G\mu|\D}$.
\end{cor}

\begin{proof}
If $f \in \Po^t(\mu|\D)$, then a straightforward computation gives the following estimate for the Taylor coefficients of $f$:
\begin{gather*}
\Big| \frac{f_n}{(n+1)^{1+\alpha}} \Big| \simeq \Big| \int_\D f\conj{z^n}(1-|z|^2)^{\alpha} dA(z) \Big|
 \leq C \|f\|_{\mu|\D,t},
\end{gather*} where $C > 0$ is independent of $f$ and $n$. In the last step we used H\"older's inequality and the reverse Carleson measure assumption $\Po^{t}(\mu|\D) \hookrightarrow \Po^{t}\left( (1-|z|)^{\alpha} dA\right)$ of (B). It follows that 
\[
\frac{|f_n|^2}{(n+1)^{2+2\alpha}} \leq C^2 \|f\|^2_{\mu|\D,t}.
\] 
In particular, this implies 
\[
\sum_{n=0}^\infty \frac{|f_n|^2}{(n+1)^{4+2\alpha}} \lesssim \|f\|^2_{\mu|\D,t}. 
\] 
If $\tau = 4 + 2\alpha$ and $g \in \hil^\tau$, then it follows from Cauchy-Schwarz inequality and the above computation that 
\[ \sum_{n=0}^\infty |f_n\conj{g_n}| \leq  \|g\|_{\hil^\tau} \|f\|_{\mu|\D,t}.
\]  
Consequently, the limit in \[ \Lambda_g(f) = \lim_{r \to 1-} \sum_{n=0}^\infty  f_n \conj{g_n}r^{2n} = \lim_{r \to 1-} \int_\T f(r\zeta)\conj{g(r\zeta)} d\m(\zeta) \] exists and moreover we have $|\Lambda_g(f)| \leq \|g\|_{\hil^\tau} \|f\|_{\mu|\D,t}.$ The existence of a solution $G$ of the equation $g = C_{G\mu|\D}$ is now an immediate consequence of \thref{CauchyDualChar}. 
\end{proof}

\section{Permanence of inner factors under norm convergence} 
\label{permanencesec}

In this section, we shall use results of Section \ref{Cauchydualsec} to establish the property (P) for the situation described in \thref{InnerPermanenceTheorem} of Section \ref{introsec}. This claim will be a consequence of the two propositions which are presented below. Here, we require that $t \in [1,\infty)$.

\begin{prop} \thlabel{BCpermanence}
Assume that the measure $\mu|\D$ satisfies (B) and $\theta =  BS_\nu$, where $B$ is a Blaschke product and $\nu$ is a Beurling-Carleson singular measure, that is, $\nu$ is supported on a countable union of Beurling-Carleson sets of Lebesgue measure zero. If $\theta h_n \to f$ in the norm of $\Po^t(\mu|\D)$, where $h_n$ are bounded analytic functions and $f \in \N^+$, then $f/\theta \in \N^+$.
\end{prop}

\begin{proof}
Let $f = u/v$ be a factorization of $f$ into $u,v \in H^\infty$ with $v$ being outer. By multiplying the sequence by the bounded function $v$, we get that $\theta p_n v \to u$ in $\Po^t(\mu|\D)$. Let $\A^\infty$ denote the algebra of analytic functions in $\D$ with smooth extensions to the boundary $\T$. For any $s \in \A^\infty \cap K_\theta$, we have
 \[ 
 \int_\T u\conj{s} \, d\m = \lim_{n \to \infty} \int_\T \theta p_n v \conj{s} \,d\m =0.
 \] 
Indeed, the passage to the limit is justified by the fact that $s \in \A^\infty$ and \thref{CauchySmoothSolutions} and the integrals in the limit vanish because $s \in K_\theta = (\theta H^2)^\perp$. It was shown in \cite{smoothdensektheta} that $\A^\infty \cap K_\theta$ is dense in $K_\theta$ (if $\theta$ is a Beurling-Carleson inner function), hence we conclude that $u = \theta u_0$, with $u_0 \in H^\infty$, and consequently $f = \theta (u_0 / v )\in \theta \N^+$.
\end{proof}

The following lemma isolates a technical argument needed in the proof of \thref{Kpermanence} below. The proof of this lemma is similar to Carleson's construction in \cite{carlesonuniqueness}, but here we instead deal with Beurling-Carleson sets of positive measure. 

\begin{lemma} \thlabel{smoothcutofflemma}
Let $E$ be a Beurling-Carleson set of positive measure and $\{A_k\}$ be the arcs complementary to $E$ on $\T$, where $A_k$ is extending between points $a_k$ and $b_k$ on $\T$. Define the function 
\[ 
	\phi(\zeta) = 
	\begin{cases} 
      |(\zeta - a_k)(\zeta - b_k)|, & \hspace{5mm} \zeta \in A_k \\
      1, & \hspace{5mm} \zeta \in E.
   \end{cases}
\]
Then $\log(\phi) \in L^1(\T)$. Let $N$ be some very large positive integer and $s$ be the outer function with modulus $|s| = |\phi|^N$ on $\T$. Let $g_1, g_2$ be any two functions analytic in $\D$ which are of the form 
\[ g_l(z) = \exp \Big( - \int_E \frac{\zeta + z}{\zeta - z} d\nu_l(\zeta) \Big), \quad z \in \D, \qquad  l=1,2,
\] 
where $\nu_1, \nu_2$ are positive finite Borel measures, both concentrated on the set $E$. If $1_{\T \setminus E}$ denotes the indicator function of the set $\T \setminus E$, then the $2\pi$-periodic function $F$ defined by 
\[ F(t) = \conj{g_1(e^{it})}g_2(e^{it})s(e^{it})1_{\T \setminus E}(e^{it})
\] is in $\mathbb{C}^n(\mathbb{R})$, as long as $N = N(n)$ is chosen large enough. 
\end{lemma}

\begin{proof}
Ideas here will be borrowed from the elegant proof of Carleson's result cited above that appears in \cite[Theorem 4.4.3]{dirichletspaceprimer}. It follows easily from the Beurling-Carleson condition on $E$ that $\log(\phi) \in L^1(\T)$. For any fixed $k$, the function $(1-z/a_k)^N(1-z/b_k)^N$ is outer, and therefore we can reproduce it using the formula \[(1-z/a_k)^N(1-z/b_k)^N = \exp \Big( N \int_\T \frac{\zeta + z}{\zeta - z} \log(|(\zeta - a_k)(\zeta - b_k)| d\m(\zeta) \Big) .\] The function $s$ is the outer function with $|s| = |\phi|^N$ on $\T$, thus \[ s(z) = \exp \Big(N\int_\T \frac{\zeta + z}{\zeta - z} \log(\phi(\zeta)) d\m(\zeta) \Big). \] Note that the moduli of the two above functions coincide on $A_k$. Consequently \[s(z) = (1-z/a_k)^N(1-z/b_k)^N\exp \Big(N h_k(z) \Big) \] where \begin{equation} \label{hk} h_k(z) = \int_{\T \setminus A_k} \frac{\zeta + z}{\zeta - z} \log\Big(\frac{\phi(\zeta)}{|(\zeta - a_k)(\zeta - b_k)|} \Big) d\m(\zeta). \end{equation} This formula shows that $s$ extends analytically across each of the arcs $A_k$ (because $h_k$ does). 

Now let $z \in A_k$ and denote by $\dist{z}{E} = \min( |z - a_k|, |z - b_k|)$ the distance from $z$ to the closed Beurling-Carleson set $E$. We have that $s$ is analytic in a neighbourhood of $z$, and if (without loss of generality) we assume that $z$ is closer to $a_k$ than it is to $b_k$, then 
\begin{gather}
\label{jAkest} |s(z)| =  |(z - a_k)(z-b_k)|^N  \nonumber \leq \dist{z}{E}^N 2^N. 
\end{gather}
We will now estimate the derivatives. It is clear from the formula \eqref{hk} that 
\[ 
|h^{(j)}_k(z)| \leq C \dist{z}{E}^{-j-1}, \quad z \in A_k 
\] 
where the constant $C$ depends on $j$ but not on $k$. The formulas for $g_1$ and $g_2$ easily imply that these functions extend analytically across each arc $A_k$, and their modulus equals 1 on each such arc. Moreover, it is again easy to see that \[ |g_l^{(j)}(z)| \leq C\dist{z}{E}^{-j-1}, \qquad l=1,2,\] with constant $C$ depending on number of derivatives $j$ taken. By the product rule, it follows from these observations that \[  \Big| \Big(\frac{d}{dt}\Big)^j (\conj{g_1}g_2s)(e^{it}) \Big| \leq C_j \dist{e^{it}}{E}^{N'}, \quad e^{it} \in A_k  \] for some positive integer $N'$ and constant $C_j$ depending only on $j$, but not on $k$, as long as $N$ is much larger than $j$. Note that we have used the fact that $|\frac{d}{dt}g(e^{it})| = |g'(e^{it})|$ if $g$ is analytic. We have thus verified that any finite number of derivatives of $\conj{g_1}g_2s$ on $\T \setminus E$ tend to zero rapidly as $z \in \T \setminus E$ approaches $E$, as long as the integer $N$ is fixed large enough. It remains now to show that the derivatives of $\conj{g_1}g_2s1_{\T \setminus E}$ on $E$ are zero. If $w_0 = e^{it_0}$ belongs to $E$ and $z = e^{it}$, then for the first derivative we have 
\begin{gather*}
 \lim_{ t \to t_0} \Big|\frac{\conj{g_1(e^{it})}g_2(e^{it})s(e^{it})1_{\T \setminus E}(e^{it})}{t-t_0} \Big|  \leq \limsup_{t \to t_0} \frac{C \dist{e^{it}}{E}^{N'}}{\abs{t-t_0}} \\ \lesssim \limsup_{t \to t_0} C \dist{e^{it}}{E})^{N'-1} = 0. 
\end{gather*} 
Higher order derivatives are handled in the exact same way. The proof is complete.

\end{proof}

\begin{remark}
Using more refined techniques from \cite{taylor1970ideals} or from Lemma 7.11 in \cite{hedenmalmbergmanspaces}, one can actually construct a smooth $C^\infty$-function. However, the current form of \thref{smoothcutofflemma} is enough for our purposes.
\end{remark}

The next result extends \thref{BCpermanence} to certain measures for which the Korenblum-Roberts part does not vanish.

\begin{prop} \thlabel{Kpermanence}
Let $\mu|\T = \omega d\m$ be a measure satisfying condition (A) and is concentrated on a single Beurling-Carleson set $E$ of positive Lebesgue measure, and let $\mu|\D$ be a measure satisfying the condition (B). Suppose $\theta$ is an inner function with corresponding singular measure which is supported on $E$. If $\theta p_n \to f$ in the norm of $\Po^t(\mu)$, where $p_n$ are bounded analytic functions and $f \in \N^+$, then $f/\theta \in \N^+$.
\end{prop}

\begin{proof} Let $H$ be the outer function with boundary values which satisfy $|H| = \min(1,\omega)$ on $E$ and $|H| = 1$ on $\T \setminus E$. We note that $H$ extends analytically across $\T \setminus E$. Let $s$ be a bounded outer function, $p$ be an analytic polynomial, and consider the family of functions 
\[ 
F_p(z) = \int_{\T} \frac{\theta(\zeta)\conj{p(\zeta)s(\zeta) \zeta H(\zeta)}}{1-\conj{\zeta}z} d\m(\zeta).
\] 
The function $F_p$ is clearly the orthogonal projection of the bounded function $\theta(\zeta)\conj{p(\zeta)s(\zeta) \zeta H(\zeta)} \in L^2(d\m)$ to the Hardy space $H^2$. These functions belong to $K_\theta$. To see this, let $\ip{\cdot}{\cdot}_{L^2}$ denote the inner product on $L^2$ and let $P_+: L^2(\T) \to H^2$ be the orthogonal projection. Then, if $q$ is any analytic polynomial, we have 
\[ \ip{\theta q}{F_p}_{L^2} = \ip{\theta q}{\theta \conj{\zeta p s H}}_{L^2} = \ip{q}{\conj{\zeta p s H}}_{L^2} = 0. \] Thus $F_p \in (\theta H^2)^\perp = K_\theta$. Moreover, the functions $F_p$ form a dense subset of $K_\theta$. Indeed, assume that $f \in K_\theta$ is orthogonal to $F_p$ for all analytic polynomials $p$. Recall that $f = \theta \conj{f_0}$ on $\T$ for some $f_0 \in H^2$ with $f_0(0) = 0$. We obtain 
\[ 0 = \ip{f}{F_p}_{L^2} = \ip{\theta \conj{f_0}}{ \theta \conj{\zeta p s H}}_{L^2} = \ip{\conj{f_0}}{\conj{\zeta p sH}}_{L^2}.\] Since $sH$ is outer, it follows that $\zeta p s H$ is dense in the set of functions in $H^2$ which vanish at $0$, and so $f_0 \equiv 0$. Consequently, $f = 0$ and the functions $F_p$ are dense in $K_\theta$.

We now decompose $F_p$ as 
\begin{gather} F_p(z) = \int_{E} \frac{ \theta(\zeta)\conj{ \zeta p(\zeta) s(\zeta) H(\zeta)}}{1-\conj{\zeta}z} d\m(\zeta) \nonumber \\ + \int_{\T \setminus E} \frac{\theta(\zeta)\conj{\zeta p(\zeta) s(\zeta) H(\zeta)}}{1-\conj{\zeta}z} d\m(\zeta). \label{2ndintegral}
\end{gather} 
Note that on $E$ we have that $|H| \leq \omega$, so that $H/\omega := U \in L^\infty(E)$. Thus the first integral in the sum above can be re-written as
\[ \int_{E} \frac{ \theta(\zeta)\conj{\zeta p(\zeta) s(\zeta) U(\zeta)}}{1-\conj{\zeta}z} \omega(\zeta) d\m(\zeta). \] This is a Cauchy transform of a function in $L^{t'}(\mu|\T)$. For the second of the integrals, the choice of $s = s(N)$ as in \thref{smoothcutofflemma} implies that \[ \zeta \mapsto \theta(\zeta)\conj{\zeta p(\zeta)s(\zeta) H(\zeta)} 1_{\T \setminus E}(\zeta)\] belongs to $C^n(\T)$, if only $N = N(n)$ is chosen large enough. For an appropriate choice of $n$, the Fourier coefficients of the function will decay rapidly enough so that the integral \eqref{2ndintegral} will define a function in $\hil^{\tau}$ (see \eqref{Xbetadef} to recall the definition), for any fixed choice of $\tau > 0$. These choices will ensure that the second integral is a Cauchy transform of a function in $L^{t'}(\mu|\D)$ by \thref{CauchySmoothSolutions}, and consequently by  \thref{CauchyDualChar} the functionals $\Lambda_{F_p}$ are bounded on $\Po^t(\mu)$. 

Now assume that $\|\theta p_n - f\|_{\mu,t} \to 0$ and $f \in \N^+$. Then there exists a factorization $f = u/v$, where $u,v\in H^\infty$. It follows that $\|\theta v p_n - u\|_{\mu,t} \to 0$. The functionals $\Lambda_{F_p}$ are bounded on $\Po^t(\mu)$ and $F_p \in K_\theta$, so
\begin{align*}
0 = \lim_{n \to \infty} \int_{\T} \theta v p_n \conj{F_p} d\m = \lim_{n \to \infty} \Lambda_{F_p}(\theta v p_n) = \Lambda_{F_p} (u) = \int_\T u \conj{F_p} d\m.
\end{align*} By density of the functions $F_p$ inside $K_\theta$, we deduce that $u \in (K_\theta)^\perp = \theta H^2$, and consequently $f/\theta \in \N^+$. 
\end{proof}

It follows easily from \thref{Kpermanence} that the permanence property holds for general measure $\mu|\T$ concentrated on a countable union $\cup_k E_k$ of Beurling-Carleson sets $E_k$ of positive measure, and $\theta$ being an inner function with corresponding singular measure concentrated on $\cup_k E_k$. Indeed, if $\theta = BS_\nu$, then we simply apply the above proposition to $\theta_{E_k} = BS_{\nu|E_k}$, where $\nu|E_k$ is the restriction of $\nu$ to the set $E_k$, for each $k$. This observation completes the proof of \thref{InnerPermanenceTheorem} of Section \ref{introsec}.

%

\section{Cyclic inner functions in $\Po^t(\mu)$}

\label{cyclicsection}

In this final section, we will establish our results on cyclicity. Here we may assume that $t \in (0,\infty)$. The notation $\|f\|_{\mu,t}$ should in case $t < 1$ be interpreted in the usual way which makes it into a translation-invariant metric on $\Po^t(\mu)$, thus for any $t>0$, we have
\[ \norm{f}_{\mu,t} = \left( \int_{\cD} \abs{f(z)}^t d\mu(z) \right)^{ \min \{ 1, 1/t\}}.
\] 
The relevant assumptions on $\mu$ that we shall need in this section are that $\mathcal{P}^t(\mu)$ can be identified as a space of analytic functions and that condition (B) holds. In fact, we may actually relax condition (B) slightly, by requiring that the measure $\mu|\D$ is such that the norms of monomials do not decrease too slowly: There exists a constant $\beta >0$, such that
\begin{enumerate}
\item[(B')]
\begin{equation*}\label{condB'}
\norm{z^n}_{\mu|\D,t} \lesssim n^{-\beta}
\end{equation*}
\end{enumerate}
for all positive integers $n$. Our main task in this section is to prove \thref{cyclicitytheorem}.
The idea behind the proof is very much inspired by the beautiful techniques of Roberts \cite{roberts1985cyclic}, which establishes cyclicity of Korenblum-Roberts inner functions in the case of Bergman-type spaces. However, due to the fact that our measure $\mu$ lives on $\T$, these techniques require certain adaptations to our setting. In this section, we shall denote by $\Po$ the set of analytic polynomials and fix $\rho$ to be the distance between the closed sets $\supp{\nu}$ and $F := \supp{\mu|\T}$, where $\nu$ is a Korenblum-Roberts measure concentrated on the complement of $F$. The following observation will allow us to assume that $\rho > 0$.

\begin{lemma} \thlabel{lemreduc}
Let $S_\nu$ be a singular inner function as in the statement of \thref{cyclicitytheorem}. Let $\nu_n$ be the restriction of $\nu$ to the set $\{ \zeta \in \T : \text{dist}(\zeta, F) > 1/n \}$ and $S_{\nu_n}$ be the corresponding inner function. If $S_{\nu_n}$ is a cyclic vector in $\Po^t(\mu)$ for all sufficiently large $n \geq 1$, then so is $S_\nu$. 
\end{lemma}

\begin{proof}
For any analytic polynomial $p$ and integer $n\geq 1$, we have 
\[ 
\begin{aligned}
\norm{ S_{\nu}p -1}_{\mu,t} \leq &  \norm{S_{\nu}p -S_{(\nu - \nu_n)}}_{\mu,t} + \norm{S_{(\nu-\nu_n)} -1}_{\mu,t}  \\
& {} \leq \norm{S_{\nu_n}p-1}_{\mu,t} + \norm{S_{(\nu - \nu_n)}-1}_{\mu,t}. 
\end{aligned}
\] 
By our assumption, we can for any sufficiently large $n \geq 1$ find an analytic polynomial $p_n$ making the quantity $\norm{S_{\nu_n}p_n-1}_{\mu,t}$ arbitrarily small, thus
\[
\inf_{p \in \Po} \norm{S_\nu p -1}_{\mu, t} \leq \norm{ S_{(\nu-\nu_n)} -1}_{\mu,t}.
\]
Now since $\nu$ is a singular continuous measure, we have that $\nu_n \to \nu$ in total variation norm, as $n \to \infty$. From this it follows that $S_{(\nu-\nu_n)}(z) \to 1$ for every  $z \in \D$. Moreover, we also have that
\[
\int_{\T} \abs{1-S_{(\nu-\nu_n)}(\zeta) }^2 dm(\zeta) = 2\left( 1- e^{ - (\nu - \nu_n)(\T) } \right) \to 0,
\]
as $n \to \infty$. This means that we can extract a subsequence $\{n_k\}_k$ with $n_k \to \infty$, such that $S_{(\nu-\nu_{n_k})}(\zeta) \to 1$ for $m$-a.e every $\zeta \in \T$. Now combining these observations and using the fact that $d\mu = d\mu \lvert \D + \omega dm$ is a finite measure on $\cD$, we may apply the dominated convergence theorem to obtain
\[ \lim_{k \to \infty} \norm{S_{(\nu-\nu_{n_k})} -1}_{\mu,t} = 0.
\]  
It follows now that $S_\nu$ is a cyclic vector in $\Po^t(\mu)$.
\end{proof}
With \thref{lemreduc} at hand, we may fix $\rho > 0$. Next we shall use a decomposition result for singular measures which implicitly appears in the work of Roberts in \cite{roberts1985cyclic}. This result is also carefully treated in \cite{ivrii2019}, where Ivrii clarifies some technical points in the construction. Before we present this result we recall the notion of modulus of continuity of a positive measure $\nu$, defined by  \[ \omega_\nu ( \delta) = \sup_{|I|< \delta} \nu(I),
\]
where the supremum is taken over all arcs $I\subset \T$ of Lebesgue measure $|I| < \delta$.

\begin{prop}[Roberts, \cite{roberts1985cyclic}] \thlabel{Robthm} 

Let $\nu$ be a Korenblum-Roberts measure and let $N>1$ be a positive integer and $c>0$ be given. Then there exists a sequence of positive measures $\{\nu_k\}_k$, satisfying the following conditions on their modulus of continuity 

\[ \omega_{\nu_k} \left( \frac{1}{n_k} \right) \leq \frac{c\log(n_k)}{n_k} , \qquad n_k = 2^{2^{N+k}}, \, k=0,1,2, \dots,
\]

such that $\nu$ decomposes as

\[ \nu = \sum_{k=1}^{\infty} \nu_k.
\]

\end{prop}


In the similar spirit as \cite{roberts1985cyclic}, we will utilize the following quantitative form of the Corona theorem.

\begin{lemma}[The Corona Theorem] There exists a universal constant $K>0$, such that whenever $f_1,f_2 \in H^\infty$ with $\| f_j \|_\infty \leq 1$, for $j=1,2$, and 
\[ \abs{f_1 (z)} + \abs{f_2(z)} \geq \delta , \qquad  z \in \D,
\]
where $0< \delta \leq 1/2 $, then there exists functions $g_1, g_2 \in H^\infty$ with $\|g_j \|_\infty \leq \delta^{-K}$, so that
\[ f_1 g_1 + f_2g_2 =1.
\]
\end{lemma}
The point of using the Corona theorem is to find bounded analytic functions $g_k$, such that $S_{\nu_{k}}g_k -1$ has very small $\Po^t(\mu)$-norm, where $\nu_k$ is a singular measure appearing in the decomposition of Proposition \ref{Robthm}.

Our next quest is to establish that singular inner functions associated to singular measures $\nu_k$ are bounded below on considerably large parts of the unit disc $\D$. To this end, we shall now fix $\nu$ to be a Korenblum-Roberts measure with the property that the distance between $\supp \nu$ and $F:= \supp{ \mu |\T }$ is $\rho>0$. Associated to $\nu$, we introduce the following sets

\[ \Omega_{n} := D_n \cup \{ z \in \D:  \text{dist}(z,\supp{\nu})\geq \rho/2 \},
\] 
where $D_n$ denotes the disc centered at origin of radius $1-1/n$. 

\begin{lemma} \thlabel{S0below}
There exists a constant $c_1 >0$, such that whenever $\nu_0$ is a proper restriction of the Korenblum-Roberts measure $\nu$ and $n_0$ is a sufficiently large positive integer such that $\omega_{\nu_0} (1/n_0) \leq c  \log(n_0)/n_0$, for some $c>0$, then the associated singular inner function $S_{\nu_0}$ satisfies 
\[ \abs{S_{\nu_0} (z) } \geq n_0^{-cc_1}, \quad z \in \Omega_{n_0}.
\]
\end{lemma}

\begin{proof} 
The bound from below on $D_{n_0}$, for some $c_1 > 0$, follows from the assumption on the modulus of continuity of $\nu_0$ and holds for any $n_0 > 2$, see \cite[Lemma 2.2]{roberts1985cyclic}). The function $S_{\nu}$ extends analytically across $\{ z \in \T : \text{dist}(z,\supp{\nu})\geq \rho/2 \}$ with $|S_{\nu}| = 1$ there, and since $\nu_0$ is a restriction of the measure $\nu$, we have $|S_{\nu_0}| \geq |S_\nu|$ on $\D$. Thus if $n_0$ is large enough, then $|S_{\nu_0}| > 1/2$ on $\Omega_{n_0} \setminus D_{n_0}$, which completes the proof.
\end{proof}


We shall now proceed by constructing an appropriate companion to the singular inner function $S_{\nu_0}$ appearing in \thref{S0below}, which will allow us to apply the Corona theorem. Indeed, we shall construct a sequence of analytic functions in the closed unit ball of $H^\infty$, which are bounded from below on $\D \setminus \Omega_n$, but have comparatively smaller $\Po^{t}(\mu)$-norms. The lemma below is technical and revolves around choosing specific parameters in a careful way.


\begin{lemma} \thlabel{h_n} 
Let $c_1>0$ denote the constant in \thref{S0below}, $K>0$ denote the constant in the Corona theorem above and fix a small $c>0$, so that $3cc_1 K <  \beta  $, where $\beta>0$ is as in (B'). Then there exists a positive integer $N_0$ and a sequence of functions $\{f_n\}_{n \geq N_0}$ inside the closed unit ball of $H^\infty$, satisfying the estimates

\begin{enumerate}

\item[(i)] 
\[ \abs{f_n (z)} \geq n^{-cc_1}, \quad z \in  \D \setminus \Omega_n ,
\]

\item[(ii)] 
\[ \norm{f_n }_{\mu, t} \leq n^{-3cc_1 K}.
\]
\end{enumerate}

\end{lemma}

\begin{proof}

Note that according to assumption (B') and $3cc_1 K < \beta$, we see that for all sufficiently large $n$, the monomials $z^{n}$ will satisfy the estimate
\begin{equation}  \label{MoN}
\norm{z^{n}}_{\mu|\D , t} \leq  \frac{1}{2} n^{-3cc_1 K}.
\end{equation}
Let $\sigma >0$ be a constant to be specified later 
and define the outer function  
\[ 
h_n (z) = \exp \left(  -\sigma n \log (n)  \int_F \frac{\zeta + z}{\zeta - z} dm(\zeta) \right), \quad z\in \D.
\]
Then $\abs{h_n} = n^{-\sigma n}$ $m$-a.e on $F$ and $\abs{h_n}=1$ on $\T \setminus F$. We now set $f_n(z) := z^{n} h_{n}(z)$ and claim that $\{f_n\}$ is our candidate, that will eventually satisfy both $(i)$ and $(ii)$. In light of \eqref{MoN}, we note that in order for $\{f_n\}$ to satisfy $(ii)$, it suffices to find a positive integer $N_0$, such that 
\begin{equation}\label{iired}
\left(\int_{F} \abs{h_n(\zeta)}^{t}\omega(\zeta) dm(\zeta) \right)^{\min(1,1/t)} \leq  \frac{1}{2} n^{-3cc_1 K}, \quad n \geq N_0.
\end{equation}
However, since $\abs{h_n} = n^{-\sigma n}$ $m$-a.e on $F$ and $\omega \in L^1(\T, m)$, it readily follows that for any fixed $\sigma >0$, we can choose $N_0 >0$ (depending on $\sigma$) such that \eqref{iired} holds for all $n \geq N_0$. It now remains to check that $\{f_n \}$ also eventually satisfies $(i)$. Recall that by definition of $\rho >0$, every $z \in \D \setminus \Omega_n$ must necessarily satisfy $\text{dist}(z,F) \geq \rho/2$, hence 
\[ - \log \abs{h_n (z)} = \sigma n \log ( n ) \int_{F} \frac{1-|z|^2}{\abs{\zeta - z}^2} dm(\zeta) \leq  \frac{8\sigma}{\rho^2} \log (n), \quad z\in \D \setminus \Omega_n.
\] 
This implies that 
\begin{equation*} \label{Lbdd}
 \abs{f_n (z)} = \abs{z}^{n} \abs{h_n (z)} \geq (1-1/n)^n n^{ -8\sigma / \rho^2} \geq 4^{-1}n^{- 8\sigma / \rho^2}, \quad z \in \D \setminus \Omega_n.
\end{equation*}
In the last step, we used the elementary inequality $(1-1/n)^n \geq 1/4$, for all $n \geq 1$. Thus in order to satisfy condition $(i)$, we only need the simple inequality
\[
n^{cc_1} \geq 4n^{8\sigma / \rho^2} 
\] 
to hold for all large $n$. However, choosing a fixed $0 < \sigma < cc_1\rho^2 /8$, it follows that the inequality above is easily met, thus $(i)$ holds for all sufficiently large $n$. Combining with our previous observations about $\{h_n\}$ eventually satisfying \eqref{iired}, we conclude that there exists an integer $N_0>0$, such that the sequence $\{f_n\}_{n\geq N_0}$ belongs to the closed unit ball of $H^\infty$ and satisfies both conditions $(i)$ and $(ii)$.
\end{proof}


We now turn to the final proof of \thref{cyclicitytheorem}. We shall follow the ideas in \cite{roberts1985cyclic} rather closely, but there are some minor differences, mainly attributed to the presence of \thref{h_n} and the fact that a part of our measure $\mu$ lives on $\T$. For the reader's convenience, we shall include a proof.

\begin{proof}[Proof of \thref{cyclicitytheorem}] 

Recall the definition of $\rho$ as the distance between $\supp{\nu}$ and $F$, which according to \thref{lemreduc} can be assumed positive and fixed. Let $c>0$ according to \thref{h_n} and let $n_0$ and $N_0$ denote the constants from \thref{S0below} and \thref{h_n}, respectively. We now pick an arbitrary integer $N> \max( n_0, N_0)$ and apply \thref{Robthm} with parameters $c$ and $N$ as above, in order to decompose our Korenblum-Roberts measure $\nu= \sum_{j=1}^\infty \nu_j$, where each measure $\nu_j$ satisfies 
\[
\omega_{\nu_j} \left(\frac{1}{n_j} \right) \leq  \frac{c \, \log n_j }{n_j}
\]
with $n_j = 2^{2^{N+j}}$. Note that according to \thref{S0below}, we can find a constant $c_1>0$, such that for any $j$, we have 
\[ \abs{S_{\nu_j}(z) } \geq n_j^{-cc_1}, \quad z\in \Omega_{n_j}.
\]
We may also apply \thref{h_n}, in order to obtain a sequence of functions $\{f_{n_j}\}_{j} \subset H^\infty$ with $\norm{f_{n_j}}_{\infty}\leq 1$, such that 
\[ \abs{f_{n_j} (z)} \geq n_j^{-cc_1}, \quad z \in \D \setminus \Omega_{n_j}.
\]
Combining, we have that $\abs{S_{\nu_j}(z)} + \abs{f_{n_j}(z)} \geq n_j^{-cc_1}$ for all $z\in \D$. Invoking the Corona theorem for each $j$, we obtain functions $g_j,h_j \in H^\infty$ with $\norm{g_j}_\infty , \norm{h_j}_\infty \leq n_j^{cc_1 K}$, that solve the equation $S_{\nu_{j}}g_j + f_{n_j}h_j =1$ on $\D$. Using this and $(ii)$ of \thref{h_n}, we get
\begin{equation}\label{S0 est} 
\norm{S_{\nu_j}g_j -1}_{\mu,t} \leq \norm{f_{n_j}}_{\mu,t} \cdot \norm{ h_j}_{\infty} \leq n_j^{-3cc_1K} \cdot n_j^{cc_1K} =  n_j^{-2cc_1K}.
\end{equation}
With this observation at hand, we now combine:
\begin{gather}
\norm{S_{\nu_1 + \nu_2}g_1 g_2 -1}_{\mu,t} \leq \norm{S_{\nu_1}g_1}_{\infty} \norm{S_{\nu_2}g_2 -1}_{\mu,t} + \norm{S_{\nu_1}g_1 -1}_{\mu,t}  \nonumber \\ 
 \leq \left(n_1 / n^2_2 \right)^{cc_1K}+ n_1^{-2cc_1K}. \nonumber
\end{gather}

Setting  $\nu^M := \sum_{1\leq j \leq M}\nu_j$ and iterating this procedure, we actually get for any positive integer $M>1$:
\[
\inf_{p \in \Po} \, \norm{S_{\nu^M}p-1}_{\mu,t} \leq \frac{1}{n_{1}^{2cc_1 K}} + \sum_{j=2}^{M} \left( \frac{n_1 \dots n_{j-1}}{n^2_j} \right)^{cc_1 K} \leq C2^{-Ncc_1K }.
\]
Here $C>0$ is an absolute constant, independent of $M,N$. With this at hand, we deduce that for any $M>1$:

\begin{gather}
\inf_{p \in \Po} \norm{S_{\nu}p -1}_{\mu, t} \leq  \norm{S_{(\nu-\nu^M)} - 1}_{\mu,t} + \inf_{p \in \Po} \norm{S_{\nu^M}p -1}_{\mu, t}  \nonumber \\ \leq \norm{S_{(\nu-\nu^M)} - 1}_{\mu,t} + C2^{-Ncc_1K}. \label{nu^M}
\end{gather}
An argument similar to that of the proof of \thref{lemreduc}, allows us to extract a subsequence $\{M_j\}_j$, such that 
\[
\lim_{j \to \infty} \norm{S_{(\nu-\nu^{M_j})} - 1}_{\mu,t} = 0.
\]
Returning back to \eqref{nu^M} with $M= M_j$ and taking $j \to \infty$, we obtain 
\[
\inf_{p \in \Po} \norm{S_{\nu}p -1}_{\mu, t} \leq C2^{-Ncc_1 K}.
\]
Since the parameter $N>1$ in the decomposition of \thref{Robthm} can be chosen arbitrary large, we finally conclude that $S_\nu$ is a cyclic vector for the shift operator on $\Po^t(\mu)$.


\end{proof}

\bibliographystyle{siam}
\bibliography{mybib}

\Addresses

\end{document}